\documentclass[a4paper,twoside,11pt,reqno]{amsart}

\usepackage[T1]{fontenc}
\usepackage[utf8]{inputenc}
\usepackage{lmodern}
\usepackage{amsmath,amssymb,mathtools}
\usepackage{microtype}
\usepackage[colorlinks=true,linkcolor=blue,citecolor=blue,urlcolor=blue]{hyperref}

\numberwithin{equation}{section}

\theoremstyle{plain}
\newtheorem{theorem}{Theorem}[section]
\newtheorem{proposition}[theorem]{Proposition}
\newtheorem{lemma}[theorem]{Lemma}

\theoremstyle{remark}
\newtheorem{remark}[theorem]{Remark}
\newtheorem{example}[theorem]{Example}

\newcommand{\F}{\mathbb F}
\newcommand{\N}{\mathbb N}
\newcommand{\Nzero}{\mathbb N_0}
\newcommand{\C}{\mathbb C}
\newcommand{\M}{\widehat{\mathbb F_n}}
\newcommand{\1}{\mathbf 1}
\newcommand{\Ent}{\operatorname{Ent}}
\newcommand{\Tr}{\operatorname{Tr}}

\newcommand{\Aadm}{\mathcal A}
\newcommand{\Aadmn}{\mathcal A_0}
\newcommand{\PiH}{\Pi}

\title[Free group hypercontractivity]{Sharp hypercontractivity for free group von Neumann algebras}

\author[X. Xie]{Xinyuan Xie}
\address{(X.X.) Department of Mathematics, University of California, Irvine, CA 92697, USA}
\email{xinyuax7@uci.edu}

\author[H. Zhang]{Haonan Zhang}
\address{(H.Z.) Department of Mathematics, University of South Carolina, Columbia, SC 29208, USA}
\email{haonanzhangmath@gmail.com}

\hypersetup{
  pdftitle={Sharp hypercontractivity for free group von Neumann algebras},
  pdfauthor={Xinyuan Xie and Haonan Zhang}
}

\date{\today}


\begin{document}

\begin{abstract}
In this paper, we settle the problem of optimal hypercontractivity for free group von Neumann algebras. Namely, for $n\ge 2$ and the free group $\mathbb{F}_n$ on $n$ generators, we prove that for any $1<p\le q<\infty$, the Poisson semigroup $P_t$ associated with the word-length function satisfies 
\[
\|P_t:L_p(\widehat{\mathbb{F}_n})\to L_q(\widehat{\mathbb{F}_n})\|\le 1 \qquad\textnormal{ if and only if }\qquad
t\ge \frac{1}{2}\log \frac{q-1}{p-1}.
\]
The main idea is to apply a refined cubic majorant estimate from a recent work of Frank and Ivanisvili \cite{FrankIvanisvili2026} to the equivalent logarithmic Sobolev inequality, and use the Haagerup-type cancellation estimate \cite{Haagerup1979}. Similar ideas and techniques extend to free products
\[
        G=\left(*_{\alpha\in A}\mathbb Z\right)*\left(*_{\beta\in B}\mathbb Z_2\right)
\]
and the free Gaussian von Neumann algebras. In the former setting, partial sharp estimates were previously obtained by Junge--Palazuelos--Parcet--Perrin--Ricard \cite{JungePalazuelosParcetPerrinRicard2015}; in the latter, our approach recovers Biane's free hypercontractivity theorem \cite{Biane1997}.
\end{abstract}

\maketitle

\section{Introduction}

Let $\F_n$ be the free group on $n$ generators, where $n\ge 2$, and let $|g|$ denote the reduced word length of $g\in\F_n$. We write $\M$
for the group von Neumann algebra generated by the left regular representation $\lambda:\F_n\to B(\ell_2(\F_n))$. It is equipped with the canonical tracial state $\tau$ given by
\[
        \tau(\lambda_g)=\delta_{g,e}.
\]
The associated noncommutative $L_p$-spaces are denoted by $L_p(\M)$, with
\[
        \|x\|_p=\tau(|x|^p)^{1/p},\qquad 1\le p<\infty.
\]
We refer to \cite{PisierXu2003} for more background about noncommutative $L_p$-spaces.
The Poisson semigroup $(P_t)_{t\ge0}$ is defined by
\[
        P_t(\lambda_g)=e^{-t|g|}\lambda_g,\qquad g\in\F_n.
\]
We write $P_t=e^{-tL}$, where
\[
        L(\lambda_g)=|g|\lambda_g.
\]
Since the word length is conditionally negative definite \cite{Haagerup1979}, $(P_t)$ is a symmetric completely positive, trace-preserving, unital semigroup. So it is contractive over $L_p(\M)$ for all $p\ge 1$. It is natural to ask if it is hypercontractive. 

The main result is the following sharp hypercontractive estimate.

\begin{theorem}\label{thm:hc}
Let $1<p\le q<\infty$. Then
\[
        \|P_t x\|_q\le \|x\|_p,\qquad x\in L_p(\M),
\]
holds if and only if
\[
        t\ge \frac12\log\frac{q-1}{p-1}.
\]
\end{theorem}

Some partial results were known before, and optimal estimates were obtained in several restricted forms. Junge, Palazuelos, Parcet and Perrin proved optimal $L_2$-$L_q$ estimates for even integer $q$ large enough \cite{JungePalazuelosParcetPerrin2017}. Junge, Palazuelos, Parcet, Perrin and Ricard obtained optimal estimates when restricted to symmetric subalgebras \cite{JungePalazuelosParcetPerrinRicard2015}. Ricard and Xu proved the optimal $L_2$-$L_q$ estimate for free groups for all $q\ge 4-\varepsilon_0$, for some $\varepsilon_0>0$, using a noncommutative martingale convexity inequality \cite{RicardXu2016}. The current work proves the optimal estimates in full range. 

The hypercontractive estimate is known to be equivalent to the log-Sobolev inequality; we use this standard equivalence and omit its proof. For a positive operator $a$ affiliated with $\M$, put
\[
        \Ent(a)=\tau(a\log a)-\tau(a)\log\tau(a).
\]
It remains to prove the equivalent log-Sobolev inequality with sharp constants. 

\begin{theorem}\label{thm:lsi}
For every positive $x\in\C[\F_n]$,
\begin{equation}\label{eq:lsi}
        \Ent(x^2)\le 2\tau(xLx).
\end{equation}
The inequality extends by closure to the natural form domain of $L^{1/2}$.
\end{theorem}

The proof has two ingredients. The first one is a relaxed form of the cubic majorant in the recent work of Frank and Ivanisvili \cite{FrankIvanisvili2026} on the sharp log-Sobolev inequality on finite cycles: For every $t\ge0$ and $\lambda>0$,
\begin{equation}\label{eq:relaxed-cubic-intro}
        2t^2\log t
        \le
        \frac{2}{3\lambda}t^3+(2\log\lambda+1)t^2-2\lambda t+\frac{\lambda^2}{3}.
\end{equation}
This helps to reduce the entropy estimates to cubic estimates that are easier to handle. The original cubic majorant corresponds to $\lambda=1$ (actually, they are equivalent as we shall see later).

The second ingredient is a third-moment estimate coming from our work \cite{XieZhang2026} following Frank and Ivanisvili. For $u=u^*\in\C[\F_n]$ with $\tau(u)=0$, we will need to control $\tau(u^3)$. However, similar estimates in \cite{XieZhang2026} on cyclic groups do not work for free groups. 
Define $S=S(u)$ and $Q=Q(u)$ by
\begin{equation}\label{eq:intro-SQ}
        S^2=\tau(u^2),
        \qquad
        Q^2=\tau(u(L-I)u).
\end{equation}
We prove the third-moment estimate (weaker than that in \cite{XieZhang2026})
\begin{equation}\label{eq:third-moment-intro}
        |\tau(u^3)|\le 3S^2Q+Q^3.
\end{equation}
The extra term $Q^3$ is the relaxation that accommodates the branching behavior of free groups. It is necessary, as we shall see below. To conclude the proof, we shall combine this weaker third-moment estimate and the above relaxed cubic majorant \eqref{eq:relaxed-cubic-intro}. Indeed, to prove \eqref{eq:lsi} for $x=\tau(x)+u$, we shall take $\lambda=\tau(x)+Q$, which depends on $x$.  

Sections \ref{sect2} and \ref{sect3} are devoted to the proofs of ingredients and the main theorem. In Section \ref{sect4}, we discuss more examples to which our framework applies.  

\subsection*{Acknowledgments}
H.Z. is supported by NSF DMS-2453408. He is grateful to Professor Quanhua Xu for bringing this problem to him in 2016 and for continued encouragement. The authors acknowledge the use of ChatGPT 5.5 Pro. All mathematical arguments and proofs in the final manuscript were checked and written by the authors.

\section{The third-moment estimate}\label{sect2}

For $k\in\Nzero=\{0,1,...\}$, let
\[
        H_k=\operatorname{span}\{\lambda_g:\ |g|=k\}\subset L_2(\M)
\]
and let $\PiH_k$ denote the $L_2$-orthogonal projection onto $H_k$. Thus every $x\in\C[\F_n]$ has the orthogonal decomposition
\[
        x=\sum_{k\ge0}x_k,
        \qquad
        x_k=\PiH_kx\in H_k.
\]
If $x=x^*$, then
\[
        \|x\|_2^2=\sum_{k\ge0}\|x_k\|_2^2,
        \qquad
        \tau(xLx)=\sum_{k\ge0}k\|x_k\|_2^2.
\]

We shall use the following admissible sets of shell lengths:
\begin{align*}
        \Aadmn
        =\{(i,j,k)\in\Nzero^3:\ |i-j|\le k\le i+j,\ i+j+k\in 2\Nzero\},\qquad 
        \Aadm
        =\Aadmn\cap\N^3.
\end{align*}

The following lemma is essentially from the proof of the Haagerup inequality \cite{Haagerup1979}. We repeat the argument here for completeness.

\begin{lemma}[Cancellation projection estimate]\label{lem:projection}
Let $a_i\in H_i$ and $b_j\in H_j$. If $(i,j,k)\notin\Aadmn$, then
\[
        \PiH_k(a_i b_j)=0.
\]
If $(i,j,k)\in\Aadmn$, then
\begin{equation}\label{eq:projection-estimate}
        \|\PiH_k(a_i b_j)\|_2\le \|a_i\|_2\|b_j\|_2.
\end{equation}
\end{lemma}

\begin{proof}
Write
\[
        a_i=\sum_{|w|=i}\alpha_w\lambda_w,
        \qquad
        b_j=\sum_{|v|=j}\beta_v\lambda_v.
\]
A product of reduced words of lengths $i$ and $j$ can reduce to length $k$ only by cancelling
\[
        d=\frac{i+j-k}{2}
\]
letters. This is possible only when $(i,j,k)\in\Aadmn$.

Assume $(i,j,k)\in\Aadmn$. A pair $w,v$ contributing to $\PiH_k(a_i b_j)$ has a unique decomposition
\[
        w=rs,
        \qquad
        v=s^{-1}t,
\]
where
\[
        |s|=d,
        \qquad
        |r|=i-d,
        \qquad
        |t|=j-d,
\]
and
\[
        |rs|=i,
        \qquad
        |s^{-1}t|=j,
        \qquad
        |rt|=k.
\]
The last condition is exactly that no further cancellation occurs between $r$ and $t$. Conversely, every such triple gives a contribution to the $k$-th shell. Since $|r|$ and $|t|$ are fixed and $|rt|=k$, the map $(r,t)\mapsto rt$ is injective on the outer summation set. Hence
\[
        \|\PiH_k(a_i b_j)\|_2^2
        =
        \sum_{\substack{|r|=i-d,\ |t|=j-d\\ |rt|=k}}
        \left|
        \sum_{\substack{|s|=d\\ |rs|=i,\ |s^{-1}t|=j}}
        \alpha_{rs}\beta_{s^{-1}t}
        \right|^2.
\]
By Cauchy-Schwarz, the inner square is bounded by $A_rB_t$, where
\[
        A_r=
        \sum_{\substack{|s|=d\\ |rs|=i}}
        |\alpha_{rs}|^2,
        \qquad
        B_t=
        \sum_{\substack{|s|=d\\ |s^{-1}t|=j}}
        |\beta_{s^{-1}t}|^2.
\]
Indeed, this only enlarges the two $s$-sums on the right-hand side. Therefore
\begin{align*}
        \|\PiH_k(a_i b_j)\|_2^2
        &\le
        \sum_{\substack{|r|=i-d,\ |t|=j-d\\ |rt|=k}} A_rB_t  \\
        &\le
        \sum_{|r|=i-d,\ |t|=j-d} A_rB_t \\
        &=
        \left(\sum_{|r|=i-d}A_r\right)
        \left(\sum_{|t|=j-d}B_t\right).
\end{align*}
Finally,
\[
        \sum_{|r|=i-d}A_r
        =
        \sum_{\substack{|r|=i-d,\ |s|=d\\ |rs|=i}}
        |\alpha_{rs}|^2
        =
        \sum_{|w|=i}|\alpha_w|^2
        =\|a_i\|_2^2,
\]
by the unique prefix-suffix decomposition of a reduced word, and similarly
\[
        \sum_{|t|=j-d}B_t
        =
        \sum_{\substack{|s|=d,\ |t|=j-d\\ |s^{-1}t|=j}}
        |\beta_{s^{-1}t}|^2
        =
        \sum_{|v|=j}|\beta_v|^2
        =\|b_j\|_2^2.
\]
This proves \eqref{eq:projection-estimate}.
\end{proof}

\begin{proposition}[Third-moment estimate]\label{thm:third-moment}
Let $u=u^*\in\C[\F_n]$ satisfy $\tau(u)=0$. Write
\[
        u=\sum_{k\ge1}u_k,
        \qquad
        u_k\in H_k,
\]
and put $b_k=\|u_k\|_2$. Define $S,Q\ge0$ by
\begin{equation}\label{eq:def-SQ}
        S^2=\sum_{k\ge1}b_k^2,
        \qquad
        Q^2=\sum_{k\ge1}(k-1)b_k^2.
\end{equation}
Then
\begin{equation}\label{eq:third-moment}
        |\tau(u^3)|\le 3S^2Q+Q^3.
\end{equation}
\end{proposition}

\begin{proof}
By Lemma \ref{lem:projection},
\[
        |\tau(u_i u_j u_k)|
        =|\tau(\PiH_k(u_i u_j)u_k)|
        \le b_i b_j b_k
\]
when $(i,j,k)\in\Aadm$, and the term is zero otherwise. Hence
\begin{equation}\label{eq:admissible-sum}
        |\tau(u^3)|
        \le
        \sum_{(i,j,k)\in\Aadm} b_i b_j b_k.
\end{equation}
Every triple $(i,j,k)\in\Aadmn$ is uniquely of the form
\begin{equation}\label{eq:rst-param}
        (i,j,k)=(s+t,r+t,r+s),
        \qquad r,s,t\in\Nzero.
\end{equation}
Indeed,
\[
        r=\frac{j+k-i}{2},
        \qquad
        s=\frac{i+k-j}{2},
        \qquad
        t=\frac{i+j-k}{2}.
\]
Set $b_0=0$. Then \eqref{eq:admissible-sum} gives
\begin{equation}\label{eq:rst-sum}
        |\tau(u^3)|
        \le
        \sum_{r,s,t\ge0}b_{s+t}b_{r+t}b_{r+s}.
\end{equation}
Since $b_0=0$, the right-hand side of \eqref{eq:rst-sum} splits as
\begin{equation}\label{eq:T1T2}
        3T_1+T_2,
\end{equation}
where
\[
        T_1=\sum_{r,s\ge1}b_r b_s b_{r+s}
\]
and
\[
        T_2=\sum_{r,s,t\ge1}b_{s+t}b_{r+t}b_{r+s}.
\]

First estimate $T_1$. Let
\[
        c_m=\sum_{r+s=m}b_r b_s,
        \qquad m\ge2.
\]
Then
\[
        T_1=\sum_{m\ge2}b_m c_m.
\]
By Cauchy-Schwarz,
\[
        T_1
        \le
        \left(\sum_{m\ge2}(m-1)b_m^2\right)^{1/2}
        \left(\sum_{m\ge2}\frac{c_m^2}{m-1}\right)^{1/2}.
\]
For each $m\ge2$,
\[
        c_m^2
        \le
        (m-1)\sum_{r=1}^{m-1}b_r^2b_{m-r}^2.
\]
Therefore
\[
        \sum_{m\ge2}\frac{c_m^2}{m-1}
        \le
        \sum_{r,s\ge1}b_r^2b_s^2
        =S^4.
\]
Since $\sum_{m\ge2}(m-1)b_m^2=Q^2$, we get
\begin{equation}\label{eq:T1}
        T_1\le S^2Q.
\end{equation}

Next estimate $T_2$. Let
\[
        K=(b_{r+s})_{r,s\ge1}.
\]
This is a finite-rank real symmetric Hankel operator, since the sequence $(b_k)$ has finite support. Then
\[
        T_2=\Tr(K^3).
\]
Moreover,
\[
        \|K\|_{\mathrm{HS}}^2
        =\sum_{r,s\ge1}b_{r+s}^2
        =\sum_{k\ge2}(k-1)b_k^2
        =Q^2.
\]
Thus
\begin{equation}\label{eq:T2}
        T_2
        \le |\Tr(K^3)|
        \le \|K\|_{S_3}^3
        \le \|K\|_{\mathrm{HS}}^3
        =Q^3.
\end{equation}
Combining \eqref{eq:T1T2}, \eqref{eq:T1}, and \eqref{eq:T2} proves \eqref{eq:third-moment}.
\end{proof}

\begin{remark}[Sharpness and necessity of the $Q^3$ term]
\label{rem:third-moment-sharpness}
The coefficients in \eqref{eq:third-moment} are optimal uniformly over the
family of free groups, and the $Q^3$ term cannot be omitted on
$\mathbb{F}_n$ for any fixed $n\ge2$.  For $1\le n<\infty$ and $k\ge1$, put
\[
        d_k=\#\{g\in\mathbb{F}_n:|g|=k\}
        =2n(2n-1)^{k-1},
        \qquad
        h_k=d_k^{-1/2}\sum_{|g|=k}\lambda_g.
\]
For $0<z<1$, let
\[
        u_{z,N}=\sum_{k=1}^{N}z^kh_k,
        \qquad
        \gamma=\frac{z^2}{1-z^2}\in(0,\infty).
\]
Orthogonality and $Lh_k=kh_k$ give, as $N\to\infty$,
\[
        S(u_{z,N})^2\longrightarrow\gamma,
        \qquad
        Q(u_{z,N})^2\longrightarrow\gamma^2.
\]
A direct reduced-word count, using the parametrization
\eqref{eq:rst-param}, gives
\[
        \tau(h_rh_sh_{r+s})=\alpha_n,
        \qquad
        \tau(h_{s+t}h_{r+t}h_{r+s})=\beta_n,
\]
for $r,s,t\ge1$, where
\begin{align*}
        \alpha_n
        &:=\frac{2n(2n-1)^{r+s-1}}{(d_r d_s d_{r+s})^{1/2}}
        =\sqrt{\frac{2n-1}{2n}}, \\
        \beta_n
        &:=\frac{2n(2n-2)(2n-1)^{r+s+t-2}}{(d_{s+t}d_{r+t}d_{r+s})^{1/2}}
        =\frac{2n-2}{\sqrt{(2n-1)(2n)}}.
\end{align*}
Consequently,
\[
        \tau(u_{z,N}^3)
        \longrightarrow
        3\alpha_n\gamma^2+\beta_n\gamma^3.
\]
For every fixed $n\ge2$, one has $\beta_n>0$.  Letting
$\gamma\to\infty$ therefore shows that no estimate of the form
\[
        |\tau(u^3)|\le \kappa S^2Q
\]
can hold on $\mathbb{F}_n$ with a finite constant $\kappa$; hence the
$Q^3$ term is necessary.

More generally, if
\[
        |\tau(u^3)|\le \kappa S^2Q+\mu Q^3
\]
holds for every $n$ with constants independent of $n$, then the preceding
family gives
\[
        3\alpha_n+\beta_n\gamma\le \kappa+\mu\gamma.
\]
Letting $\gamma\downarrow0$ and then $n\to\infty$ yields
$\kappa\ge3$, while letting $\gamma\to\infty$ and then
$n\to\infty$ yields $\mu\ge1$.  Thus the coefficients $3$ and $1$ in
\eqref{eq:third-moment} are optimal uniformly over all finite $n$, and
consequently also on $\mathbb{F}_\infty$.

When $n=1$, one has $\beta_1=0$, reflecting the absence of interior
branching.  In this case $\mathbb{F}_1=\mathbb{Z}$, and the stronger
estimate without the $Q^3$ term holds; see \cite{XieZhang2026}.
\end{remark}

\section{The relaxed cubic majorant and proof of the main theorem}\label{sect3}

\begin{lemma}[Relaxed cubic majorant]\label{lem:relaxed-majorant}
For every $t\ge 0$ and every $\lambda>0$,
\begin{equation}\label{eq:relaxed-majorant}
        2t^2\log t
        \le
        \frac{2}{3\lambda}t^3+(2\log\lambda+1)t^2-2\lambda t+\frac{\lambda^2}{3},
\end{equation}
where $2t^2\log t$ is interpreted as $0$ at $t=0$.
\end{lemma}

\begin{proof}
Recall the usual cubic majorant proved in \cite{FrankIvanisvili2026}
\begin{equation}\label{eq:usual-majorant}
        2s^2\log s
        \le
        \frac23(s-1)^2(s+2)+(s^2-1),
        \qquad s\ge0,
\end{equation}
or equivalently
\[
        2s^2\log s\le \frac23s^3+s^2-2s+\frac13.
\]
Apply this inequality with $s=t/\lambda$ and multiply by $\lambda^2$. This gives
\[
        2t^2\log\frac{t}{\lambda}
        \le
        \frac{2}{3\lambda}t^3+t^2-2\lambda t+\frac{\lambda^2}{3}.
\]
Adding $2t^2\log\lambda$ to both sides yields \eqref{eq:relaxed-majorant}.
\end{proof}

\begin{remark}
    Taking $\lambda=t>0$, \eqref{eq:relaxed-majorant} becomes an equality. So we actually proved the following variational formula:
    \begin{equation}
                2t^2\log t
        =\min_{\lambda >0}\left\{
        \frac{2}{3\lambda}t^3+(2\log\lambda+1)t^2-2\lambda t+\frac{\lambda^2}{3}\right\},\qquad t>0.
    \end{equation}
\end{remark}

\begin{proof}[Proof of Theorem \ref{thm:lsi}]
By homogeneity, it suffices to consider $x\ge0$ with $\tau(x^2)=1$. Write
\[
        x=a\1+u,
        \qquad
        a=\tau(x),
        \qquad
        \tau(u)=0.
\]
Since $x\ge0$ and $\tau(x^2)=1$, we have $0<a\le1$. Define $S,Q\ge0$ by
\[
        S^2=\tau(u^2)=1-a^2,
        \qquad
        Q^2=\tau(u(L-I)u),
\]
and set
\[
        E=\tau(xLx)=\tau(uLu)=S^2+Q^2,
        \qquad
        \lambda=a+Q.
\]
By Proposition \ref{thm:third-moment},
\[
        \tau(u^3)\le 3S^2Q+Q^3.
\]
Therefore
\begin{align*}
        \tau(x^3)
        &=a^3+3aS^2+\tau(u^3)  \\
        &\le a^3+3a(1-a^2)+3(1-a^2)Q+Q^3 \\
        &=3\lambda-3a\lambda^2+\lambda^3.
\end{align*}
Apply Lemma \ref{lem:relaxed-majorant} to $x$ by functional calculus. Since $\tau(x^2)=1$,
\[
        \Ent(x^2)=\tau(2x^2\log x).
\]
Thus
\begin{align*}
        \Ent(x^2)
        &\le \frac{2}{3\lambda}\tau(x^3)+(2\log\lambda+1)-2a\lambda+\frac{\lambda^2}{3} \\
        &\le \frac{2}{3\lambda}(3\lambda-3a\lambda^2+\lambda^3)
              +(2\log\lambda+1)-2a\lambda+\frac{\lambda^2}{3} \\
        &=\lambda^2-4a\lambda+3+2\log\lambda.
\end{align*}
On the other hand,
\[
        2E=2(S^2+Q^2)=2(1-a^2+Q^2).
\]
Using $\lambda=a+Q$, we get
\[
        2E-\bigl(\lambda^2-4a\lambda+3+2\log\lambda\bigr)
        =\lambda^2-1-2\log\lambda.
\]
Finally,
\[
        \lambda^2-1-2\log\lambda\ge0,
        \qquad \lambda>0,
\]
with equality only at $\lambda=1$. Hence
\[
        \Ent(x^2)\le2E=2\tau(xLx).
\]
This proves \eqref{eq:lsi} for positive elements of $\C[\F_n]$. The extension to the form domain follows by standard truncation and approximation, since $\C[\F_n]$ is a form core for $L^{1/2}$.
\end{proof}

\section{Extension and discussion}\label{sect4}

The proof for the free groups can be extended to more examples, which we briefly discuss here. The aim is not to give the most general possible statement, but to highlight the flexibility and mention several examples. 

\begin{theorem}[Cubic reduction]\label{thm:cubic-reduction}
Let $(\mathcal N,\tau)$ be a finite tracial von Neumann algebra, and let $P_t=e^{-tL}$ be a symmetric Markov semigroup on $\mathcal N$. Assume that $L\1=0$ and that the spectral gap is normalized to one, namely
\[
        \tau(uLu)\ge \tau(u^2)
\]
for every self-adjoint $u$ in the form domain of $L^{1/2}$ with $\tau(u)=0$. Let $\mathcal D$ be a unital $*$-subalgebra which is a form core for $L^{1/2}$.

Suppose that for every self-adjoint $u\in\mathcal D$ with $\tau(u)=0$, if
\[
        S^2=\tau(u^2),
        \qquad
        Q^2=\tau(u(L-I)u),
\]
then
\begin{equation}\label{eq:cubic-reduction-assumption}
        \tau(u^3)\le 3S^2Q+Q^3.
\end{equation}
Then for all positive $x\in \mathcal D$, the following log-Sobolev inequality holds
\begin{equation}\label{eq:cubic-reduction-lsi}
        \Ent(x^2)\le 2\tau(xLx).
\end{equation}
\end{theorem}

\begin{proof}
 The proof is the same as in the proof of Theorem \ref{thm:lsi}. 
  \end{proof}

The following proposition gives a sufficient condition for \eqref{eq:cubic-reduction-assumption}.

\begin{proposition}[Triangular shell criterion]\label{prop:triangular-shell}
Assume that
\[
        L_2(\mathcal N,\tau)=\bigoplus_{k\ge0}H_k,
        \qquad
        H_0=\C\1,
        \qquad
        L|_{H_k}=k\textnormal{Id},
\]
and that each $H_k$ is invariant under $*$. Let $\Pi_k$ be the orthogonal projection onto $H_k$. Suppose that
\[
        \Pi_k(H_iH_j)=0
\]
unless
\[
        |i-j|\le k\le i+j,
        \qquad
        i+j+k\in2\Nzero,
\]
and that, whenever this condition holds,
\begin{equation}\label{eq:triangular-shell-projection}
        \|\Pi_k(ab)\|_2\le \|a\|_2\|b\|_2,
        \qquad
        a\in H_i,
        \quad
        b\in H_j.
\end{equation}
Then \eqref{eq:cubic-reduction-assumption} holds. Hence Theorem \ref{thm:cubic-reduction} applies.
\end{proposition}

\begin{proof} The proof is the same as the free group case. We repeat it here with a slightly different presentation.
Let $u=u^*$, $\tau(u)=0$, and write
\[
        u=\sum_{k\ge1}u_k,
        \qquad
        u_k\in H_k.
\]
Put $b_k=\|u_k\|_2$ and $b_0=0$. By \eqref{eq:triangular-shell-projection},
\[
        |\tau(u_i u_j u_k)|
        =|\langle \Pi_k(u_i u_j),u_k^*\rangle|
        \le b_i b_j b_k
\]
on the triangle-parity set, and the term is zero outside it. Hence
\[
        |\tau(u^3)|
        \le
        \sum_{(i,j,k)\in\Aadm}b_i b_j b_k.
\]
Parameterize the triangle-parity triples by
\[
        (i,j,k)=(s+t,r+t,r+s),
        \qquad
        r,s,t\in\Nzero.
\]
Then
\begin{equation}\label{eq:full-hankel-sum}
        |\tau(u^3)|
        \le
        \sum_{r,s,t\ge0}b_{s+t}b_{r+t}b_{r+s}.
\end{equation}
We now write the last triple sum as the cubic moment of one block Hankel matrix.  Let
\[
        v=(b_r)_{r\ge1},
        \qquad
        K=(b_{r+s})_{r,s\ge1},
\]
and set
\[
        \mathsf H=
        \begin{pmatrix}
        0 & v^T\\
        v & K
        \end{pmatrix}
\]
on $\C\oplus\ell_2(\N)$.  Then
\[
        \Tr(\mathsf H^3)
        =
        \sum_{r,s,t\ge0}b_{s+t}b_{r+t}b_{r+s}
        =3\langle Kv,v\rangle+\Tr(K^3).
\]
Moreover,
\[
        \|v\|_2^2=S^2,
        \qquad
        \|K\|_{\mathrm{HS}}^2
        =\sum_{r,s\ge1}b_{r+s}^2
        =\sum_{m\ge2}(m-1)b_m^2
        =Q^2,
\]
where
\[
        S^2=\sum_{k\ge1}b_k^2,
        \qquad
        Q^2=\sum_{k\ge1}(k-1)b_k^2.
\]
Therefore
\[
        \langle Kv,v\rangle
        \le \|K\|\,\|v\|_2^2
        \le \|K\|_{\mathrm{HS}}S^2
        =S^2Q,
\]
and
\[
        \Tr(K^3)
        \le |\Tr(K^3)|
        \le \|K\|_{\mathrm{HS}}^3
        =Q^3.
\]
Combining these estimates with \eqref{eq:full-hankel-sum} gives
\[
        |\tau(u^3)|\le 3S^2Q+Q^3.\qedhere
\]

\end{proof}

\begin{example}[Tree-like free products]\label{ex:tree-like-groups}
Consider the free product
\[
        G=\left(*_{\alpha\in A}\mathbb Z\right)*\left(*_{\beta\in B}\mathbb Z_2\right)
\]
of some copies of $\mathbb Z$ and some copies of $\mathbb Z_2$. Let $|g|$ be the associated reduced word length, and let
\[
        P_t(\lambda_g)=e^{-t|g|}\lambda_g.
\]
This includes free products of groups whose Cayley graphs, with the chosen generators, are trees. Some sharp hypercontractivity estimates for $G$ are already known by  Junge, Palazuelos, Parcet, Perrin and Ricard \cite{JungePalazuelosParcetPerrinRicard2015}.  We show that $G$ satisfies Proposition \ref{prop:triangular-shell}.

Indeed, the Cayley graph is a tree.  Hence a product of two reduced words of lengths $i$ and $j$ can change length only by cancelling a terminal segment of the first word against the inverse initial segment of the second.  Therefore, a product can land in $H_k$ only if
\[
        |i-j|\le k\le i+j,
        \qquad
        i+j+k\in2\Nzero.
\]
Moreover, the same prefix-suffix argument as in Lemma \ref{lem:projection} gives
\[
        \|\Pi_k(ab)\|_2\le \|a\|_2\|b\|_2,
        \qquad
        a\in H_i,
        \quad
        b\in H_j.
\]
Thus Proposition \ref{prop:triangular-shell} proves the third-moment estimate \eqref{eq:cubic-reduction-assumption}.  Theorem \ref{thm:cubic-reduction} then gives
\[
        \Ent(x^2)\le 2\tau(xLx).
\]
\end{example}

\begin{example}[Free Gaussian]\label{ex:free-gaussian}
Let $\mathcal H_{\mathbb R}$ be a real Hilbert space and let
$\Gamma(\mathcal H_{\mathbb R})$ be the free Gaussian von Neumann algebra
on the full Fock space
\[
        \mathcal F(\mathcal H_{\mathbb C})
        =\mathbb C\Omega\oplus\bigoplus_{k\ge1}\mathcal H_{\mathbb C}^{\otimes k}.
\]
Let $W(\xi)$ denote the free Wick word associated with
$\xi\in\mathcal H_{\mathbb C}^{\otimes k}$.  The $k$-th homogeneous space is
\[
        H_k=\{W(\xi):\xi\in\mathcal H_{\mathbb C}^{\otimes k}\},
        \qquad k\ge0,
\]
with $H_0=\mathbb C\1$, and the Ornstein-Uhlenbeck generator is
\[
        L W(\xi)=k W(\xi),
        \qquad \xi\in\mathcal H_{\mathbb C}^{\otimes k}.
\]
We verify Proposition \ref{prop:triangular-shell}, and then this recovers
Biane's free hypercontractivity theorem \cite{Biane1997}.

The reason is that free Wick multiplication has the same triangular
cancellation structure as reduced-word multiplication in a free group.  If
$\xi\in\mathcal H_{\mathbb C}^{\otimes i}$ and
$\eta\in\mathcal H_{\mathbb C}^{\otimes j}$, then the free Wick product
formula gives
\begin{equation}\label{eq:free-wick-product}
        W(\xi)W(\eta)
        =\sum_{r=0}^{\min(i,j)} W(\xi\frown_r\eta),
\end{equation}
where $\xi\frown_r\eta\in\mathcal H_{\mathbb C}^{\otimes(i+j-2r)}$
is the $r$-fold contraction (\cite[Definition 1.21 and Proposition 1.25, Eq. (1.7)]{KNPS12} and \cite[Lemma 3.2]{HoudayerRicard2011}).
Also, we have \cite[Eq. (1.5)]{KNPS12}
\[
        \|W(\xi)\|_2=\|\xi\|,\qquad\xi\in H_{\mathbb C}^{\otimes k},
\]
 and the contraction is
contractive:
\[
        \|\xi\frown_r\eta\|\le \|\xi\|\,\|\eta\|.
\]
Consequently, for $a\in H_i$ and $b\in H_j$, the projection $\Pi_k(ab)$ is zero unless
\[
        k=i+j-2r
\]
for some $0\le r\le \min(i,j)$.  Equivalently,
\[
        |i-j|\le k\le i+j,
        \qquad
        i+j+k\in2\mathbb N_0.
\]
When this condition holds, the integer $r=(i+j-k)/2$ is unique, and \eqref{eq:free-wick-product} gives
\[
        \Pi_k(W(\xi)W(\eta))=W(\xi\frown_r\eta).
\]
Therefore
\[
        \|\Pi_k(W(\xi)W(\eta))\|_2
        =\|\xi\frown_r\eta\|
        \le \|\xi\|\,\|\eta\|
        =\|W(\xi)\|_2\,\|W(\eta)\|_2.
\]
Thus \eqref{eq:triangular-shell-projection} is satisfied and Proposition \ref{prop:triangular-shell} gives
\[
        |\tau(u^3)|\le 3S^2Q+Q^3
\]
for every self-adjoint mean-zero Wick polynomial $u$.  Theorem \ref{thm:cubic-reduction} then yields
\[
        \Ent(x^2)\le 2\tau(xLx).
\]
\end{example}

We close the paper with the following remark.
    The third moment estimate \eqref{eq:cubic-reduction-assumption} can fail when a special structure appears, and let us explain it for discrete group von Neumann algebras as it is already addressed in \cite{JungePalazuelosParcetPerrin2017}. Let $G$ be a discrete group equipped with a symmetric conditionally negative length function $|\cdot|$, and consider the associated Markov semigroup $e^{-tL}$ satisfying $L\lambda_g=|g|\lambda_g$.  Let us normalize the spectral gap $\inf_{g\neq e}|g|=1$ so that the log-Sobolev constant is expected to be 2. Suppose that a 3-loop exists, that is, there exist $g_1,g_2,g_3$ such that $g_1 g_2 g_3=e$ and $|g_1|=|g_2|=|g_3|=1$. Then the right-hand side of \eqref{eq:cubic-reduction-assumption} vanishes for $u=\sum_{1\le j\le 3}a_j(\lambda_{g_j}+\lambda_{g_j}^\ast)$. However, the left-hand side of \eqref{eq:cubic-reduction-assumption} can be strictly positive for this $u$. This is not surprising. As already pointed out in \cite{JungePalazuelosParcetPerrin2017}, the presence of 3-loops forces the optimal hypercontractivity time to be strictly larger than the usual one. In particular, the log-Sobolev constant should be strictly larger than 2, so the current machinery of third-moment estimate does not work.

\end{document}